\documentclass[12pt]{article}
\usepackage{amsmath}
\usepackage{amsfonts}
\usepackage{amssymb}
\usepackage[letterpaper,top=2cm,bottom=2cm,left=3cm,right=3cm,marginparwidth=1.75cm]{geometry}
\usepackage{amsmath}
\usepackage{graphicx}
\usepackage[colorlinks=true, allcolors=blue]{hyperref}
\usepackage{tikz}

\usepackage{comment}
\usepackage{xcolor}
\usepackage{color}

\usepackage{graphicx}%
\newfont{\footsc}{cmcsc10 at 8truept}
\newfont{\footbf}{cmbx10 at 8truept}
\newfont{\footrm}{cmr10 at 10truept}
\newtheorem{theorem}{Theorem}

\newtheorem{claim}[theorem]{Claim}

\newtheorem{corollary}[theorem]{Corollary}

\newtheorem{lemma}[theorem]{Lemma}

\newtheorem{proposition}[theorem]{Proposition}
\newtheorem{remark}[theorem]{Remark}

\newenvironment{proof}[1][Proof.]{\noindent{\emph {#1}  }}  {\hfill$\Box$\bigskip}

\def\blfootnote{\xdef\@thefnmark{}\@footnotetext}
\begin{document}

\title{Bounding the sum of the largest\\ signless Laplacian eigenvalues of a graph}

\author{Aida Abiad
	\thanks{Department of Mathematics and Computer Science, Eindhoven University of Technology, The Netherlands}
	\thanks{Department of Mathematics: Analysis, Logic and Discrete Mathematics, Ghent University, Belgium} 
	\thanks{Department of Mathematics and Data Science, Vrije Universiteit Brussel, Belgium (\texttt{a.abiad.monge@tue.nl})} 
	\and 
	Leonardo de Lima\thanks{Graduate Program in Mathematics, Federal University of Parana, Curitiba, Brazil (\texttt{leonardo.delima@ufpr.br})}
	\and 
	Sina Kalantarzadeh\thanks{Sharif University of Technology, Iran (\texttt{sinakalantarzadehhh@yahoo.com})}
	\and Mona Mohammadi\thanks{Sharif University of Technology, Iran (\texttt{mona.mohammadi78@gmail.com})} \and
	Carla Oliveira\thanks{Department of
		Mathematical, National School of Statistical Sciences, Rio de Janeiro, Brazil (\texttt{carla.oliveira@ibge.gov.br})}
}

\date{}
\maketitle

%%%%%%%%%%%%%%%%%%%%%%%%%%%%%%%%%%%%%%%%%%%%%%%%%%%%%%%%%%%%%%%%%%%%%%%%%%
\begin{abstract}
We show several sharp upper and lower bounds for the sum of the largest eigenvalues of the signless Laplacian matrix. These bounds improve and extend previously known bounds.

\vspace{0.3cm}

%\noindent \textbf{Keywords: }\emph{signless Laplacian, sum of the largest eigenvalues, interlacing, sharp bounds}

\medskip
\noindent \textbf{AMS classification: }05C50, 05C35

\end{abstract}
%%%%%%%%%%%%%%%%%%%%%%%%%%%%%%%%%%%%%%%%%%%%%%%%%%%%%%%%%%%%%%%%%%%%%%%%%%

%%%%%%%%%%%%%%%%%%%%%%%%%%%%%%%%%%%%%%%%%%%%%%%%%%%%%%%%%%%%%%%%%%%%%%%%%%
\section{Introduction}
%%%%%%%%%%%%%%%%%%%%%%%%%%%%%%%%%%%%%%%%%%%%%%%%%%%%%%%%%%%%%%%%%%%%%%%%%%

Consider $G=(V,E)$ to be a simple graph with $n$ vertices such that $|V|=n$. Let $N(v_i)$ be the set of neighbors of a vertex $v_i \in V$ and $\left| N(v_i) \right|$ its cardinality. The sequence degree of $G$ is denoted by $d(G) = \left(d_1(G),d_2(G),\ldots,d_n(G) \right),$ such that $d_i(G) = |N(v_i)|$ is the degree of the vertex $v_i \in V$ and $ d_1(G) \geq d_2(G) \geq \cdots \geq d_n(G).$ The \emph{Laplacian matrix} of $G$ is defined as $L=D-A$, where $D$ is the diagonal matrix of the vertex degrees and $A$ is the adjacency matrix of $G.$ The \emph{signless Laplacian matrix} (or  \emph{$Q$-matrix}), defined as $Q = A + D$, has received a lot of attention, see, e.g., \cite{CSI, CSII, CSIII}. The eigenvalues of $L$ and $Q$ are denoted as
$\lambda_{1}(G) \geq \lambda_{2}(G) \geq \cdots \geq \lambda_{n}(G)=0$ and $q_{1}(G) \geq q_{2}(G) \geq \cdots \geq q_{n}(G)$, respectively. For simplicity, the eigenvalues of $Q$ and $L$ are called here as $Q-$eigenvalues and $L-$eigenvalues of $G,$ respectively. 
%Assuming that $G$ has $n$ vertices with degrees $d_{1}\geq
%d_{2}\geq \cdots \geq d_{n}$, it is known that, for $1\le m\le n$,

Using Schur's inequality ~\cite{S1923}, it is known that
\begin{equation}
\label{eq:1HaemersSeminar}
\sum_{i=1}^{m}\lambda_{i}(G)\geq
\sum_{i=1}^{m}d_{i}(G)
\end{equation}
and
\begin{equation}
\label{eq:1HaemersSeminar2}
\sum_{i=1}^{m}q_{i}(G)\geq
\sum_{i=1}^{m}d_{i}(G)
\end{equation}
for $1 \leq m \leq n.$ Note that if $m=n,$ we have equality in \eqref{eq:1HaemersSeminar} and \eqref{eq:1HaemersSeminar2}, because both
terms correspond to the trace of $L$ and $Q,$ respectively. 
%To prove~\eqref{eq:1HaemersSeminar} by using interlacing, let $B$ be a principal $m\times m$ submatrix of $L$ indexed by the subindeces corresponding to the $m$ higher degrees, with eigenvalues $\mu_{1}\geq \mu_{2}\geq \cdots \geq \mu_{m}$. Then,
%\begin{equation*}
%\tr B= \sum_{i=1}^{m}d_{i}=
%\sum_{i=1}^{m}\mu_{i},
%\end{equation*}
%and, by interlacing,  $\lambda_{n-m+i}\leq \mu_i\le \lambda_i$ for $i=1,\ldots,m$, whence \eqref{eq:1HaemersSeminar} follows. Similarly, reasoning with  the principal submatrix $B$ (of $L$) indexed by the $m$ vertices with lower degrees we get:
%\begin{equation}
%\label{eq:1HaemersSeminar+}
%\sum_{i=1}^{m}\lambda_{n-m+i}\leq
%\sum_{i=1}^{m}d_{n-m+i}.
%\end{equation}
An improvement of~\eqref{eq:1HaemersSeminar} is due to Grone~\cite{G1995}, who proved that if $G$ is connected and $k<n$ then,
\begin{equation}
\label{eq:Gronebound}
\sum_{i=1}^{m}\lambda_{i}(G)\geq
\sum_{i=1}^{m}d_{i}(G)+1.
\end{equation}
The first author, Fiol, Haemers and Perarnau \cite{AFHP2014} showed a generalization and a variation of (\ref{eq:Gronebound}), as well as an extension of some inequalities by
Grone and Merris \cite{GM1994}.

There have been some results bounding the sum of the two largest signless Laplacian eigenvalues, see for example \cite{AOT2013},  \cite{D2019} and \cite{YY14}. Cvetkovi\'c, Rowlinson and Simi\'c \cite{CRS07} proved that $q_1(G) \geq d_1(G)+1$. After, Das \cite{Das10} showed that  $q_2(G) \geq d_2(G)-1.$ An immediate lower bound is $q_1(G)+q_2(G) \geq d_1(G)+d_2(G)$ (note that Schur's inequality can also be used to obtain the same result).

In this paper, we use a mix of two types of interlacing (Cauchy and quotient matrix) to obtain several sharp lower and upper bounds on the sum of the largest signless Laplacian eigenvalues. In particular, we show a lower bound for $q_1(G) + q_2(G)$ and characterize the case of equality. This bound improves previously known bounds. We also show several sharp bounds for the sum of the largest $Q$-eigenvalues, providing a $Q$-analog of Grone's inequality  (\ref{eq:Gronebound}). The paper is organized such that preliminary results are presented in Section \ref{sec"preliminaries}, and Sections \ref{sec:twoevs} and \ref{sec:Section2and3extendedQ} are devoted to the main results. 

%%%%%%%%%%%%%%%%%%%%%%%%%%%%%%%%%%%%%%%%%%%%%%%%%%%%%%%%%%%%%%%%%%%%%%%%%%
\section{Preliminaries}\label{sec"preliminaries}
%%%%%%%%%%%%%%%%%%%%%%%%%%%%%%%%%%%%%%%%%%%%%%%%%%%%%%%%%%%%%%%%%%%%%%%%%%
%Define $G=(V,E)$ as a simple graph on $n$ vertices. Let $N(u)$ be the set of neighbors of a vertex $u \in V$ and $\left| N(u) \right|$ its cardinality. The sequence degree of $G$ is denoted by $d(G) = \left(d_1(G),d_2(G),\ldots,d_n(G) \right),$ such that $d_i(G) = |N(v_i)|$ is the degree of the vertex $v_i \in V$ and $ d_1(G) \geq d_2 (G)\geq \cdots \geq d_n(G).$ Write $A$ for the adjacency matrix of $G$ and let $D$ be the diagonal matrix of the row-sums of $A,$ i.e., the degrees of $G$.

%\textcolor{red}{The matrix $L\left(  G\right)  =A-D$ is called the \emph{Laplacian} or the $L$-matrix of $G$ and the matrix $Q\left(  G\right)  =A+D$ is called the \emph{signless Laplacian} or the $Q$-matrix of $G$. As usual, we shall index the eigenvalues of $L\left(  G\right)  $ and $Q\left(  G\right)  $ in non-increasing order and denote them as $\lambda_{1}(G) \geq \lambda_{2}(G) \geq \cdots \geq \lambda_{n}(G)$ and $q_{1}(G) \geq q_{2}(G) \geq \cdots \geq q_{n}(G)$, respectively. For simplicity, the eigenvalues of the matrices $L(G)$ and $Q(G)$ are often called here as the $L$-eigenvalues and $Q$-eigenvalues of $G$, respectively.} 

%Let $K_{n}$ denote the complete graph on $n$ vertices, $S_{n}$ the star graph, and $K_{n_1,n_2}$ the complete bipartite graph such that $n_1 \geq n_2$ and $n=n_1+n_2.$

Some of our proofs use a classical result in matrix theory, the Cauchy interlacing theorem (see for instance \cite[Theorem 4.3.8]{HoJo85}, \cite{H1995}).

\begin{theorem}[Interlacing Theorem] (\cite{H1995}) \label{the:interlacinghaemers}
  Let $A$ be a real symmetric $n \times n$ matrix with eigenvalues $\lambda_1 \geq \dots \geq \lambda_n$.
For some $m<n$, let $S$ be a real $n \times m$ matrix with orthonormal columns, $S^\top S = I$,
and consider the matrix $B = S^\top AS$, with eigenvalues $\mu_1 \geq \dots \geq \mu_m$. 
\begin{description}
\item[$(i)$] The eigenvalues of $B$ interlace those of $A$, that is,
\begin{equation}\label{interlacinginequalities}
    \lambda_i \geq \mu_i \geq \lambda_{n-m+i}, \qquad i = 1, \ldots, m,
\end{equation}
\item[$(ii)$] If the interlacing is tight, that is, if exist an integer $k \in [0,m]$ such that $\lambda_i = \mu_i$, for $i = 1,\ldots,k$, and
$\mu_i = \lambda_{n-m+i}$, for $i = k + 1,\ldots,m$, then $SB = AS$.
\end{description}
\end{theorem}

Two interesting types of eigenvalue interlacing appear depending on the choice of $B$: when $B$ is a principal submatrix of $A$ (the so-called Cauchy interlacing), and when $B$ is the quotient matrix of a certain partition of $A$. Our proofs in Section \ref{sec:Section2and3extendedQ}  will require a novel mix of the two types of eigenvalue interlacing.

The Cauchy interlacing theorem for the signless Laplacian matrix holds in a specific way. In  \cite[Theorem 2.6]{WB11}, for a vertex $v \in V$, the authors proved that the $Q-$eigenvalues of $G$ and $G-v$ interlace, where $G-v$ is a graph obtain from $G$ removing the vertex $v$: %Notice that the use of the interlacing theorem plays an important role in our proofs.

%\begin{theorem}
%\label{interlacing_vertex_version}
%Let $G$ be a graph on $n$ vertices and let $v$ be a vertex of $G$. Let $q_1,q_2,\ldots, q_n$ ($q_1 \geq q_2 \geq \ldots \geq q_n$) and $s_1,s_2,\ldots,s_{n-1}$ ($s_1 \geq s_2 \geq \cdots \geq s_{n-1}$) be the $Q$- eigenvalues of $G$ and $G-v$, respectively. Then
%$$0 \leq q_n \leq s_{n-1} \leq \cdots \leq s_2 \leq q_2 \leq s_1 \leq q_1.$$
%\end{theorem}
%
%\begin{theorem} [\cite{CRS07b}]
%\label{interlacing_edge_version}
%Let $G$ be a graph on $n$ vertices and let $e$ be an edge of $G$. Let $q_1,q_2,\ldots, q_n$ ($q_1 \geq q_2 \geq \cdots \geq q_n$) and $s_1,s_2,\ldots,s_n$ ($s_1 \geq s_2 \geq \cdots \geq s_n$) be the $Q$- eigenvalues of $G$ and $G-e$, respectively. Then
%$$0\leq s_n \leq q_n \leq \cdots \leq s_2 \leq q_2 \leq s_1 \leq q_1.$$
%\end{theorem}

\begin{theorem}[\cite{WB11}]\label{interlacing_vertex_version}
Let $G$ be a graph of order $n$ and $v \in V.$ Then for $i=1,\ldots,n-1,$
$$ q_{i+1}(G) -1 \leq q_{i}(G-v) \leq q_{i}(G), $$
where the right inequality holds if and only if $v$ is an isolated vertex.
\end{theorem}

Cvetkovi\'c \emph{et al.} in \cite{CRS07b} presented an edge removal version of the Cauchy interlacing theorem for the $Q$-eigenvalues by using line graphs:

\begin{theorem} [\cite{CRS07b}]
\label{interlacing_edge_version}
Let $G$ be a graph on $n$ vertices and let $e$ be an edge of $G$. Let $q_1 \geq q_2 \geq \cdots \geq q_n$ and $s_1 \geq s_2 \geq \cdots \geq s_n$ be the $Q$- eigenvalues of $G$ and $G-e$, respectively. Then
$$0\leq s_n \leq q_n \leq \cdots \leq s_2 \leq q_2 \leq s_1 \leq q_1.$$
\end{theorem}

It turns out that the Cauchy interlacing theorem also holds for the Laplacian matrix of $G$ as showed by Godsil and Royle \cite[Theorem 13.6.2]{GR01}:

\begin{theorem}[\cite{GR01}] \label{th_interlace_laplacian}
Let $G$ be a graph on $n$ vertices and let $e \in E$ be an edge of $G$. The $L-$eigenvalues of $G$ and $H = G - e$ interlace, that is,
$$ \lambda_{1}(G) \geq \lambda_1(H) \geq \cdots \geq  \lambda_{n-1}(G) \geq  \lambda_n(H) = \lambda_{n}(G) = 0.$$
\end{theorem}

Let $A$ is a symmetric real matrix whose rows and columns are indexed by $X=\{1,2, \ldots, n\}.$ Let $\{X_1, X_2, \ldots, X_n\}$ be a partition of $X$. The characteristic matrix $S$ is the $n \times m$ whose j${-th}$ column is the characteristic vector of $X_j$ $(j=1, \ldots, m).$ Define $n_i = |X_i|$ and the diagonal matrix $K=diag(n_1, \ldots, n_m).$ Let $A$ be partitioned according to $\{X_1, X_2, \ldots, X_n\}$, that is 
$$A = \left[
\begin{array}{ccc}
A_{11} &  \cdots  & A_{1m} \\
\vdots  & \ddots  &  \vdots  \\
A_{m1} &  \cdots & A_{mm} \\
\end{array}
 \right] $$
where $A_{ij}$ denotes the submatrix (block) of $A$ formed by rows in $X_i$ and the column in $X_j.$ Let $b_{ij}$ the average row sum of $A_{ij}.$ Then the matrix $B=(b_{ij})$ is called the \emph{quotient matrix}. We easily have $KB=S^TAS$ and $S^TS=K.$ If the row sum of each block $A_{ij}$ is constant the partition is called \emph{equitable}. If each vertex in $X_i$ has the same number $b_{ij}$ of neighbors in part $X_j,$ for any $j$ (or any $j \ne i)$, the partition is called \emph{almost equitable}.

\begin{lemma}({\cite{BH12}}) \label{lemma}
Let $A$ be a symmetric matrix of order $n$, and suppose $P$ is a partition of
$\{1,\ldots,n\}$ such that the corresponding partition of $A$ is equitable with quotient matrix $B$.
Then the spectrum of $B$ is a sub(multi)set of the spectrum of $A$, and all corresponding
eigenvectors of $A$ are in the column space of the characteristic matrix $C$ of $P$ (this means that the entries of the eigenvector are constant on each partition class $U_i$). The remaining eigenvectors of $A$ are orthogonal to the columns of $C$ and the corresponding
eigenvalues remain unchanged if the blocks $A_{i,j}$ are replaced by $A_{i,j} + c_{i,j}J$ for certain
constants $c_{i,j}$ (as usual, $J$ is the all-one matrix).
\end{lemma}

We will denote by $K_{n}$, $S_{n}$ and $K_{n_1,n_2}$ the complete graph, star graph and complete bipartite graph, respectively such that $n_1 \geq n_2$ and $n=n_1+n_2$.

%%%%%%%%%%%%%%%%%%%%%%%%%%%%%%%%%%%%%%%%%%%%%%%%%%%%%%%%%%%%%%%%%%%%%%%%%%
\section{Bounds on the sum of the two largest eigenvalues}\label{sec:twoevs}
%%%%%%%%%%%%%%%%%%%%%%%%%%%%%%%%%%%%%%%%%%%%%%%%%%%%%%%%%%%%%%%%%%%%%%%%%%%%
Our main result of this section is a sharp lower bound on the sum of the two largest signless Laplacian (Theorem \ref{th3}). Some preparation is required. To obtain the main result we first prove some auxiliary results for a subgraph $H$ of $G$ by considering the two vertices of the largest degrees and their neighbors. 

Let $u$ and $v$ be the vertices with the two largest degrees of a graph $G$, that is, $|N(u)|=d_1(G)$ and $|N(v)| = d_2(G)$. A subgraph $H(V_{H},E_{H})$ of $G$ can be obtained by taking the vertex set as $V_{H} = \{ v_i \in V \, | \, v_i \in N(u) \cup N(v) \cup \{u\} \cup \{v\} \}$
and the edge set as
$E_{H}= \{ (v_i,v_j) \in E \,|\, v_i \, \in \{u,v\} \mbox{  and  } v_j  \in   N(u) \cup N(v) \}$.

We improve the lower bound from \cite{CRS07}, 
$q_1(G) \geq d_1(G)+1$. Roughly speaking, the proofs of our main results in this section (Theorems \ref{th3} and \ref{th4}) follow from the fact that $q_1(G)+q_2(G) \geq q_1(H)+q_2(H)$ (by Theorems \ref{interlacing_vertex_version}, \ref{interlacing_edge_version}, \ref{th_interlace_laplacian}) and also that $d_1(G)+d_2(G) = d_1(H)+d_2(H)$ since we did not remove any vertex from $N(u)$ and $N(v)$ of $G$ to build the graph $H$. In fact, if we prove that $q_1(H)+q_2(H) \geq d_1(H)+d_2(H)+1$, we are done. Before proving that, we need to introduce some notation, several preliminary results and two types of graphs obtained by the definition of $H = (V_{H}, E_{H}).$

Let $S_1 = N(u)\setminus \left( N(v) \cup v \right)$, $S_2 = N(u) \cap N(v)$ and $S_3 = N(v)\setminus \left( N(u) \cup u \right)$, such that $|S_1| = r$, $|S_2| = p$ and $|S_3| = s$. Figure 1 displays the two possible types of graphs isomorphic to $H$. Notice that if $u$ and $v$ are  not adjacent, $H$ belongs to $\mathcal{H}(p,r,s)$ such that $d_1(G) = d_1(H) = p+r$ and $d_2(G) = d_2(H) = p+s.$ If $u$ and $v$ are adjacent, $H$ belongs  to $\mathcal{G}(p,r,s)$ such that $d_1(G) = d_1(H) = p+r+1$ and $d_2(G) = d_2(H) = p+s+1.$
%
%
%Consider $\mathcal{H}(p,r,s)$ as the family of graphs obtained from $2K_1 \vee \overline{K_{p}}$ with additional $r$ and $s$ pendant vertices to the vertices $u$ and $v$ whose have largest and second largest  degree, respectively, as shown by Figure  1. The Propositions \ref{pr4} and \ref{pr5} show that the lower bound to $q_2(G)$ of Lemma \ref{lemmaq2} can be improved for some graphs of the family $\mathcal{H}(p,r,s).$
%
\begin{figure}[h]
\begin{center}
 \includegraphics[width=0.9\linewidth]{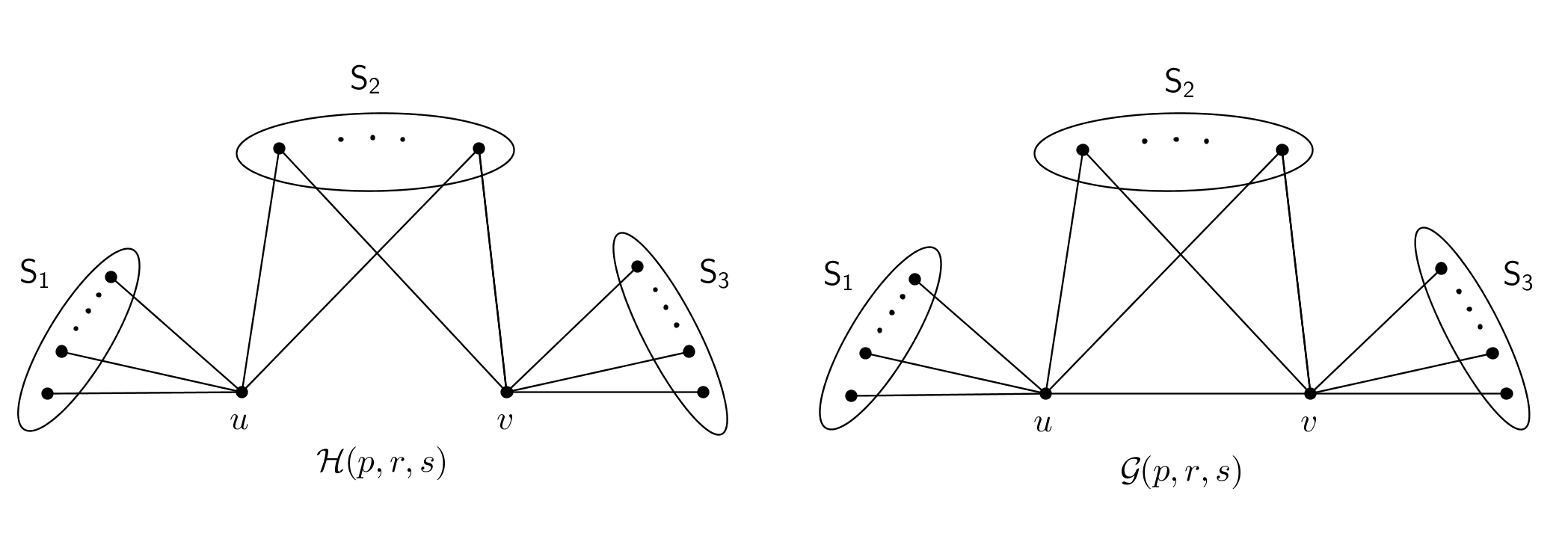}
\label{fig1}
\vspace{-0.4cm}
\caption{Families of graphs of the type $H(V_{H},E_{H})$.}
\end{center}
\end{figure}

Lemmas \ref{lemmaq1}  and \ref{lemmaq2} establish lower bounds to $q_1(G)$ and
$q_2(G)$ in terms of $d_1(G)$ and $d_2(G)$ that will be useful for our purposes here.

\begin{lemma}[\cite{CRS07}]\label{lemmaq1}
Let $G$ be a connected graph on $n \geq 4$ vertices. Then, $$q_{1}(G) \geq d_{1}(G)+1$$
with equality if and only if $G$ is the star $S_{n}.$
\end{lemma}

\begin{lemma}[\cite{Das10}]\label{lemmaq2}
Let $G$ be a graph. Then $$q_{2}(G) \geq d_{2}(G) -1.$$
\end{lemma}

Next, we improve the lower bounds of the previous lemmas for all graphs in $\mathcal{H}(p,r,s)$ and $\mathcal{G}(p,r,s)$. This will be crucial to later prove our main result.

\begin{proposition}\label{pr4}
For $p \geq 1$ and $r \geq s \geq 1,$  let $G \in \mathcal{H}(p,r,s)$ be a graph on $n \geq 3$ vertices. Then $$q_2(G) >
 d_2(G).$$
\end{proposition}

\begin{proof}
For $p \geq 1$, $r \geq s \geq 1$, consider $G \in \mathcal{H}(p,r,s).$ Labeling the vertices in a convenient way, we get
\[ Q(G) = \left( \begin{array}{c|c|c|c|c}
 p+r   &     0      &  \mathbf{1}_{1 \times p}   & \mathbf{1}_{1 \times r} & \mathbf{0}_{1 \times s}   \\  \hline
0       &    p+s    &   \mathbf{1}_{1 \times p}    & \mathbf{0}_{1 \times r}   & \mathbf{1}_{1 \times s} \\  \hline
 \mathbf{1}_{p \times 1} & \mathbf{1}_{p \times 1} & 2\mathbf{I}_{p \times p} & \mathbf{0}_{p \times r} & \mathbf{0}_{p \times s} \\ \hline
 \mathbf{1}_{r \times 1} & \mathbf{0}_{r \times 1} & \mathbf{0}_{r \times p} & \mathbf{I}_{r \times r} & \mathbf{0}_{r \times s} \\ \hline
 \mathbf{0}_{s \times 1} & \mathbf{1}_{s \times 1} & \mathbf{0}_{s \times p} & \mathbf{0}_{s \times r} & \mathbf{I}_{s \times s} \\
\end{array}\right ).
\]

Observe that $\mathbf{x}_{j} = e_{3}-e_{j}$, for $j=4,\ldots,p+2$ are eigenvectors associated to the eigenvalue $2$ which has multiplicity at least $p-1.$ Also, let us define $\mathbf{y}_{j} = e_{p+3}-e_{j}$ for each $j=p+4,\ldots,p+r+2$ and $\mathbf{z}_{j} = e_{p+r+3}-e_{j}$ for each $j=p+r+4,\ldots,p+r+s+2.$ Observe that $\mathbf{y}_{j}$ and $\mathbf{z}_{j}$ are eigenvectors associated to the eigenvalue $1$ with multiplicity at least $r+s-2.$ The  remaining $5$ eigenvalues are the same of the quotient matrix $M$ according to the equitable partition of $Q(G)$:
$$M = \left(
\begin{array}{ccccc}
p+r &  0   &   p    &   r  & 0 \\
0  & p+s  &    p    &    0  &  s \\
1 &    1 &    2 &     0  & 0\\
1 &   0  &   0   & 1    &  0 \\
0 &   1  &  0   &   0   & 1 \\
\end{array}
 \right) .$$

The characteristic polynomial of M is given by $f(x,p,r,s) = x^{5}-(s+r+p+6)x^4+((r+p+3)s+(p+3)r+p^2+6p+5)x^{3}+((-2r-p-4)s+(-2p-2)r-2p^2-6p-2)x^{2}+((s+r)p+p^{2}+2p)x.$

Considering $r=s+k$, where $k \geq 0$ note that $f(d_2(G),p,s+k,s)=f(p+s,p,s+k,s) = (s+p) (ks^2 + s^2 + 2kps + 2p - 2ks - 2s + kp^2 - kp) > 0.$ As $f(0,p,r,s)<0$, if we take $q_2(G) < y < q_1(G),$ then $f(y,p,r,s)<0.$ So, since $f(d_2(G),p,r,s) > 0$ and from Lemma \ref{lemmaq1}, we get $q_2(G) > d_{2}(G)$.
\end{proof}

\begin{proposition}\label{pr5}
For $p \geq 1$ and \ $r \geq 1,$ let $G \in \mathcal{H}(p,r,0)$ be a graph on $n \geq 3$ vertices. Then $$q_2(G) \geq d_2(G).$$ Equality holds if and only if $G=P_4$.
\end{proposition}

\begin{proof}
For $p,r \geq 1,$ consider $G \in \mathcal{H}(p,r,0).$ Labeling the vertices in a convenient way, we get

\[ Q(G) = \left( \begin{array}{c|c|c|c}
 p+r   &     0      &  \mathbf{1}_{1 \times p}   & \mathbf{1}_{1 \times r} \\  \hline
0       &    p    &   \mathbf{1}_{1 \times p}    & \mathbf{0}_{1 \times r}  \\  \hline
 \mathbf{1}_{p \times 1} & \mathbf{1}_{p \times 1} & 2\mathbf{I}_{p \times p} & \mathbf{0}_{p \times r}  \\ \hline
 \mathbf{1}_{r \times 1} & \mathbf{0}_{r \times 1} & \mathbf{0}_{r \times p} & \mathbf{I}_{r \times r}  \\
\end{array}\right ).
\]

If $p=r=1$, then $q_2(G) = d_2(G)=2$. If $p \geq 1$ and $r \geq 2$, observe that $\mathbf{x}_{j} = e_{3}-e_{j}$, for $j=4,\ldots,p+2$ are eigenvectors associated to the eigenvalue $2$ which has multiplicity at least $p-1.$

Let us define $\mathbf{y}_{j} = e_{p+3}-e_{j}$ for each $j=p+4,\ldots,p+r+2.$ Observe that $\mathbf{y}_{j}$ are eigenvectors associated to the eigenvalue $1$ with multiplicity
at least $r-1.$ The  remaining $4$ eigenvalues are the same of the quotient matrix $M$ according to the equitable partition of $Q(G):$

$$M = \left(
\begin{array}{cccc}
p+r &  0   &   p    &   r \\
0  & p  &    p    &    0  \\
1 &    1 &    2 &     0  \\
1 &   0  &   0   & 1  \\
\end{array}
 \right) .$$

The characteristic polynomial of $M$ is given by $f(x,p,r) = x^{4}+(-r-2p-3)x^3+((p+2)r+p^2+4p+2)x^{2}+(-pr-p^2-2p)x.$ As $f(-1,p,r)>0$, if we take $q_2(G) < y < q_1(G),$ then $f(y,p,r)<0.$

Note that $f(d_2(G),p,r)=f(p,p,r) = rp^2 > 0.$ Therefore, from Lemma \ref{lemmaq1}, we have $q_2(G) > d_{2}(G)$. So $q_2(G) \geq d_2(G)$ with equality if and only if $G=P_4$.
\end{proof}
%
%
%Let $\mathcal{G}(p,r,s)$ be the graph isomorphic to $\mathcal{H}(p,r,s)$ plus the edge $(u,v)$ as shown by Figure 2. The next proposition shows that Theorem \ref{th3} is true for the family $\mathcal{G}(0,r,s).$
%\begin{figure}[h]
%\begin{center}
%\includegraphics[width=2in, height=1.5in]{grafo.jpg}
%\caption{Graph $\mathcal{G}(p,r,s)$.}
%\end{center}
%\label{fig2}
%\end{figure}

\begin{proposition}\label{pr3}
For $r,s \geq 1,$ let $G \in \mathcal{G}(0,r,s).$  Then,
$$q_1(G) + q_{2}(G) > d_1(G) + d_{2}(G)+1.$$
\end{proposition}

\begin{proof}
For $r,s \geq 1,$ let $G \in \mathcal{G}(0,r,s).$  
Labeling the vertices of $G$ conveniently, we get

\[ Q(G) = \left( \begin{array}{c|c|c|c}
 r+1   &     1      &  \mathbf{1}_{1 \times r} & \mathbf{0}_{1 \times s}   \\  \hline
 1       &    s+1    &  \mathbf{0}_{1 \times r}   & \mathbf{1}_{1 \times s} \\  \hline
 \mathbf{1}_{r \times 1} & \mathbf{0}_{r \times 1} & \mathbf{I}_{r \times r} & \mathbf{0}_{r \times s} \\ \hline
 \mathbf{0}_{s \times 1} & \mathbf{1}_{s \times 1} & \mathbf{0}_{s \times r} & \mathbf{I}_{s \times s} \\
\end{array}\right ).
\]

Let us define $\mathbf{y}_{j} = e_{3}-e_{j}$ for each $j=4,\ldots,r+2,$  and $\mathbf{z}_{j} = e_{r+3}-e_{j}$
for each $j=r+4,\ldots,r+s+2.$ Observe that $\mathbf{y}_{j}$ and $\mathbf{z}_{j}$ are eigenvectors associated to the eigenvalue $1$ with multiplicity
at least $r+s-2.$ The  remaining $4$ eigenvalues are the same of the quotient matrix $M$ according to the equitable partition of $Q(G):$
$$M = \left(
\begin{array}{cccc}
r+1 &  1   &  r  &  0 \\
1     & s+1  &  0  &  s \\
1     &   0  &  1  &  0\\
0     &   1  &  0  &  1 \\
\end{array}
 \right) .$$

The characteristic polynomial of $M$ is given by $f(x,r,s) = x^4+(-r-s-4)x^3+((r+2)s +2r+5)x^2 + (-r-s-2)x.$ As $f(0,r,s)>0$, if we take $q_2(G) < y < q_1(G),$ then $f(y,r,s)<0.$
Since $d_2(G)=s+1,$ we get
$f(d_2(G),r,s)=s(s+1)(r-s) \geq 0$ which implies $q_{2}(G) \geq d_{2}(G).$ From the equality conditions of Lemma \ref{lemmaq1}, $q_1(G) > d_1(G) + 1$ and the result follows.
\end{proof}

Next, in Proposition \ref{pr2}, we present some bounds to $q_1(G)$ and $q_2(G)$ when $G \in \mathcal{G}(p,r,s)$ for $p \geq 1$, $r \geq s \geq 0.$

\begin{proposition}\label{pr2}
For $p \geq 1$, $r \geq s \geq 0,$ let $G \in \mathcal{G}(p,r,s)$ be a graph on $n \geq 3$ vertices. Then

\begin{description}
\item[$(i)$] If $r=p=1$ and $s=0,$ then $q_2(G) = d_2(G);$
\item[$(ii)$] if $p = 1$ and $r=s$, then $q_1(G) > d_1(G) + \frac{3}{2}$ and $q_2(G) > d_2(G) -\frac{1}{2};$
%\item[(ii)] if $p = 1$ and $r \geq s+1$, then $q_2(G) > d_2(G);$
\item[$(iii)$] if $p \geq 2$ and $r=s$, then $q_1(G) > d_1(G) + 2;$
\item[$(iv)$] if $p \geq 1$ and $r \geq s+3$, then $q_2(G) > d_2(G);$
\item[$(v)$] if $p \geq 1$ and $r \in \{ s+1,s+2 \},$ then $q_1(G) > d_1(G)+1+\frac{p}{n}$ and $q_2(G) > d_2(G)-\frac{p}{n}.$
\end{description}
\end{proposition}

\begin{proof}
For $p \geq 1$, $r \geq s \geq 1$, let $G \in \mathcal{G}(p,q,r).$  Labeling the vertices in a convenient way, we obtain

\[ Q(G) = \left( \begin{array}{c|c|c|c|c}
 p+r+1   &     1      &  \mathbf{1}_{1 \times p}   & \mathbf{1}_{1 \times r} & \mathbf{0}_{1 \times s}   \\  \hline
 1       &    p+s+1    &   \mathbf{1}_{1 \times p}    & \mathbf{0}_{1 \times r}   & \mathbf{1}_{1 \times s} \\  \hline
 \mathbf{1}_{p \times 1} & \mathbf{1}_{p \times 1} & 2\mathbf{I}_{p \times p} & \mathbf{0}_{p \times r} & \mathbf{0}_{p \times s} \\ \hline
 \mathbf{1}_{r \times 1} & \mathbf{0}_{r \times 1} & \mathbf{0}_{r \times p} & \mathbf{I}_{r \times r} & \mathbf{0}_{r \times s} \\ \hline
 \mathbf{0}_{s \times 1} & \mathbf{1}_{s \times 1} & \mathbf{0}_{s \times p} & \mathbf{0}_{s \times r} & \mathbf{I}_{s \times s} \\
\end{array}\right ).
\]

Observe that $\mathbf{x}_{j} = e_{3}-e_{j}$, for $j=4,\ldots,p+2$ are eigenvectors associated to the eigenvalue $2$ which has multiplicity at least $p-1.$ Also, let us define $\mathbf{y}_{j} = e_{p+3}-e_{j}$ for each $j=p+4,\ldots,p+r+2,$  and $\mathbf{z}_{j} = e_{p+r+3}-e_{j}$
for each $j=p+r+4,\ldots,p+r+s+2.$ Observe that $\mathbf{y}_{j}$ and $\mathbf{z}_{j}$ are eigenvectors associated to the eigenvalue $1$ with multiplicity
at least $r+s-2.$ The others $5$ eigenvalues are the same of the reduced matrix
$$M = \left(
\begin{array}{ccccc}
p+r+1 &  1   &   p    &   r  & 0 \\
1  & p+s+1  &    p    &    0  &  s \\
1 &    1 &    2 &     0  & 0\\
1 &   0  &   0   & 1    &  0 \\
0 &   1  &  0   &   0   & 1 \\
\end{array}
 \right) .$$

The characteristic polynomial of $M$ is given by $f(x,p,r,s) = x^{5}+(-s-r-2p-6)x^4+((r+p+4)s+(p+4)r+p^2+8p+13)x^{3}+((-2r-2p-5)s+(-2p-5)r-2p^2-14p-12)x^{2}+((p+2)s+(p+2)r+p^{2}+12p+4)x-4p.$ Since all eigenvalues of $Q$ are nonnegative, the roots of $f(x,p,r,s)$ are also nonnegative. As $f(0,p,r,s)<0$, if we take $q_2(G) < y < q_1(G),$ then $f(y,p,r,s)<0.$ This fact will be useful for the proof of the following cases below.

The largest and second largest degree of $G$ are given by $d_{1}(G) = p+r+1$ and $d_2(G) = p+s+1,$ respectively. Using the characteristic polynomial $f(x,p,r,s),$ we prove the following cases:

\begin{description}
\item[$(i)$] $p=1$ and $r=s$:  observe that $f(d_1(G)+3/2,1,r,s) = -\frac{(6r+25)(4r^2+12r+25)}{32}$ and $f(d_2(G)-1/2,1,r,s) = \frac{(2r-1)(20r^2+12r+5)}{32}.$ For $r = s\geq 1$, we get $f(d_1(G)+3/2,1,r,s) <0 $ and $f(d_2(G)-1/2,1,r,s)>0.$  As from Lemma \ref{lemmaq1}, $f(d_1(G)+1,1,r,s) < 0$ and we get $f(d_1(G)+3/2,1,r,s) <0$, so $q_1(G) > d_1(G)+3/2.$ Also, from Lemma \ref{lemmaq2}, $f(d_2(G)-1,1,r,s)>0 $ and we get $f(d_2(G)-1/2,1,r,s)>0$, then $q_2(G)>d_2(G)-1/2.$
\item[$(ii)$] $p \geq 2$ and $r=s:$ note that $\;f(d_1(G)+2,p,r,r)= -(2r+3p+6)(pr-2r+p^{2}+p-2) <0$ and also from Lemma \ref{lemmaq1}, $\;f(d_1(G)+1,p,r,r)<0.$  So, we can conclude that $q_1(G) > d_1(G)+2.$
\item[$(iii)$] $p \geq 1$ and $r \geq s+3:$ note that $f(d_2(G),p,s+k,s)=ks^3+(3k-2)ps^2+((3k-5)p^2+(k+2)p-k)s+(k-3)p^3+(k+1)p^2 > 0 $ for $k \geq 3.$  Also, from Lemma \ref{lemmaq2}, we get $f(d_2(G)-1,p,s+k,s)>0.$  So, $q_2(G) > d_2(G).$
\item[$(iv)$] $p \geq 1$ and $r \in \{ s+1,s+2 \}:$ note that $n=2s+p+r+2$. Considering first $r=s+1,$ we get
$f(d_1(G)+1+p/n,p,r,s)= f(p+s+3+p/(p+2s+3),p,s+1,s) <0$ and
$f(d_2(G)-p/n,p,r,s) = f(p+s+1-p/(p+2s+3),p,s+1,s) > 0.$ Using Lemmas \ref{lemmaq1} and \ref{lemmaq2} analogous to the previous cases, we get $q_1(G)> d_1(G)+1+p/n$ and $q_2(G) > d_2(G)-p/n.$

 Now, setting $r=s+2$ and using a computational support, we find that $f(d_1(G)+1+p/(p+2s+4),p,s+2,s) <0$ and $f(d_2(G)-p/(p+2s+4),p,s+2,s)>0.$
Using Lemmas \ref{lemmaq1} and \ref{lemmaq2} analogously to the previous cases, we get $q_1(G)> d_1(G)+1+p/n$ and $q_2(G) > d_2(G)-p/n.$ 
\end{description}
For the cases $p,r \geq 1, s=0,$ the proof is similar to the previous cases and the result follows. 
\end{proof}
 
\begin{theorem}\label{th3} Let $G$ be a simple connected graph on $n \geq 3$ vertices. Then
\[
q_{1}\left(  G\right) + q_{2}\left(  G\right)  \geq d_{1}\left(  G\right) + d_{2}\left(  G\right) + 1.
\]
Equality holds if and only if $G$ is a complete graph $K_{3}$ or a star $S_{n}.$
\end{theorem}

\begin{proof}
%\noindent \textbf{of Theorem \ref{th3}}
Let $G$ be a simple connected graph on $n \geq 3$ vertices. Assume that $u$ and $v$ are the vertices with largest and second largest degrees of $G$, i.e., $d(u) = d_{1}(G)$ and $d(v) = d_{2}(G).$
Take $H$ as a subgraph of $G$ containing $u$ and $v$ such that $H$ belongs to  either $\mathcal{H}(p,q,r)$ or $\mathcal{G}(p,r,s).$ Note that $d_1(G)+d_2(G)=d_1(H)+d_2(H)$ and from interlacing, Theorems  \ref{interlacing_vertex_version} and \ref{interlacing_edge_version}, $q_1(G)+q_2(G) \geq q_1(H)+q_2(H).$

%If $N(x)$ is the neighborhood set of the vertex $x$, define $X_1=N(u) \setminus N(v), X_2=N(v) \setminus N(u)$ and $Y=N(u) \cap N(v)$ such that $ p = |Y|, |X_1|= r \geq s =|X_2|.$
%%Set $Z = N(u) \cup N(v) \cup \{u\} \cup \{v\}$ and remove all vertices of $G$ outside $Z$. Remove all edges of the subgraph induced by $N(u)$ and also from the subgraph induced by $N(v).$
%The graph $H = (V(H),E(H))$ is defined as follows: $V(H)=X_1 \cup X_2 \cup Y \cup \{u,v\}$ and the edge set is $E(H)=\{uv_1|v_{1} \in X_1 \cup Y \} \cup \{ vv_{1}| v_{1} \in X_{2} \cup Y \} \cup \{ (u,v) \}$ or $E(H)=\{uv_1|v_{1} \in X_1 \cup Y \} \cup \{ vv_{1}| v_{1} \in X_{2} \cup Y \}$ depending whether $u$ and $v$ are adjacent or not.

Firstly, suppose that $H \in \mathcal{H}(p,r,s).$ Since $G$ is connected, the cases $p=0$ with any $r$ and $s$ are not possible.
If $p=1$ and $r=s=0,$ then $H = \mathcal{H}(1,0,0)= S_{3}$ and $q_1(H)+q_2(H)=4= d_1(H)+ d_2(H)+1.$ If $p \geq 2$ and $r=s=0,$ then $H = \mathcal{H}(p,0,0)= K_{2,p}$ and $q_1(H) +q_2(H) = 2p+2 > d_1(H)  + d_2(H)+1 = 2p+1.$ If $p,r \geq 1$ and $s=0,$ from Proposition \ref{pr5} and Lemma \ref{lemmaq1}, we get $q_1(H)+q_2(H)>d_1(H)+d_2(H)+1.$ Now, if $p \geq 1$ and $r \geq s \geq 1,$ from Proposition \ref{pr4} and Lemma \ref{lemmaq1}, follows that $q_1(H)+q_2(H)>d_1(H)+d_2(H)+1.$

Now, suppose that $H \in \mathcal{G}(p,q,r).$ If $p=s=0$ and $r \geq 1,$ $H = \mathcal{G}(0,r,0) = S_{r+2}$ and $q_1(H)+q_2(H)=r+3 = d_1(H)+d_2(H)+1.$  If $p=0$ and $r  \geq s \geq 1,$ the result follows from Proposition \ref{pr3}. If $p=1$ and $r=s=0,$ then $H$ is the complete graph $K_3$ and $q_1(H)+q_2(H)=5 = d_1(H)+d_2(H)+1.$ If $p \geq 2$ and $r=s=0$, then $H = \mathcal{G}(p,0,0) = K_2 \vee \overline{K_{p}}$, i.e., the complete split graph, and it is well-known that $q_1(H)=(n+2+\sqrt{n^2+4n-12})/2$ and $q_{2}(H)=n-2.$ It is easy to check that for $p \geq 2$, we have $q_1(H)+ q_{2}(H) > d_{1}(H) + d_{2}(H)+1.$
If $p \geq 1, r \geq s \geq 0$, from Proposition \ref{pr2} and Lemmas \ref{lemmaq1} and \ref{lemmaq2}, we get
$q_1(H)+q_2(H) \geq d_1(H) + d_2(H) +1.$

From the cases above, the equality conditions are restricted to the graphs $K_3$ and $S_n$, and the result follows.
\end{proof}

Next, we show that a more general bound such as $$\sum_{i=1}^{m} q_{i}(G) \geq  1+ \sum_{i=1}^{m} d_{i}(G)$$ does not hold for $m \geq 3$. Consider $S_{n}^{+}$ to be the graph obtained from a star $S_{n}$ plus an edge.

\begin{proposition}\label{pr1}
Let $G$ be isomorphic to $S_{n}^{+}.$ For $m \geq 3,$  $$\sum_{i=1}^{m} q_{i}(G) <  1+ \sum_{i=1}^{m} d_{i}(G).$$
\end{proposition}

\begin{proof}
Let $G$ be isomorphic to $S_{n}^{+}.$ In this case, $d_1(G)=n-1,d_2(G)=d_3(G)=2$ and $d_4(G)=\cdots=d_{n}(G)=1.$ From  \cite[Lemma 3.1]{OLRC13}, we have
$q_3(G)=\cdots=q_{n-1}(G) = 1$ and also
\begin{eqnarray}
n <& q_1(G) <& n + \frac{1}{n} \nonumber \\
3 - \frac{2.5}{n} <& q_2(G) <& 3 - \frac{1}{n}. \nonumber
\end{eqnarray}
From \cite{Das10} we know that $q_{n}(G) < d_{n}(G)$, and then we obtain
\begin{equation*}
0 \leq q_{n} (G)< 1.  
\end{equation*}

Since for $m \geq 3,$
\begin{equation*}
1+\sum_{i=1}^{m} d_{i}(G) = n+m+1, 
\end{equation*}
and also
\begin{equation*}
\sum_{i=1}^{m} q_{i}(G) < n+m+1,  
\end{equation*}
thus the result follows.
\end{proof}

Finally, we consider the inequality $\lambda_{1}(G)+\lambda_{2}(G) \geq d_1(G)+ d_{2}(G) + 1 $ by Grone \cite{G1995}, and characterize the extremal cases.

\begin{theorem}\label{th4}
Let $G$ be a connected graph on $n \geq 3$ vertices. Then
$$ \lambda_{1}(G)+\lambda_{2}(G) \geq d_1(G)+ d_{2}(G) + 1 $$
with equality if and only if $G$ is a star $S_{n}.$
\end{theorem}

\begin{proof}
%\noindent \textbf{of Theorem \ref{th4}}
Let $G$ be a simple connected graph on $ n\geq 3$ vertices. The result $\lambda_1(G)+\lambda_2(G) \geq d_1(G)+d_2(G)+1$ follows from Grone in \cite{G1995}. Now, we need to prove the equality case.
Assume that $u$ and $v$ are the vertices with largest and second largest degrees of $G$, i.e., $d(u) = d_{1}(G)$ and $d(v) = d_{2}(G).$ Let $tK_1$ be the graph on $t$ vertices and no edges. Take $H$ as a subgraph of $G$ containing $u$ and $v$ and isomorphic to either $\mathcal{H}(p,q,r) \cup (n-p-r-s-2)K_1$ or $\mathcal{G}(p,r,s) \cup (n-p-r-s-2)K_1.$ Note that $d_1(G)+d_2(G)=d_1(H)+d_2(H)$ and from interlacing Theorem \ref{th_interlace_laplacian}, $\lambda_1(G)+\lambda_2(G) \geq \lambda_1(H)+\lambda_2(H).$

Firstly, suppose that $H \in \mathcal{H}(p,q,r) \cup (n-p-r-s-2)K_1.$  In this case, $H$ is bipartite and $\lambda_{i}(H) = q_{i}(H)$ for $i=1,\ldots,n$ (see  \cite[Proposition 2.5]{CRS07}). The proof is analogous to Theorem \ref{th3} and the equality cases are similar. Then equality occurs when $H = S_{3}.$

Now, suppose that $H \in \mathcal{G}(p,r,s) \cup (n-p-r-s-2)K_1.$ If $p=1, r=s=0$ then $H = \mathcal{G}(1,0,0) \cup (n-3)K_{1} = K_3 \cup (n-3)K_{1}$  and $6 = \lambda_1(H)+\lambda_2(H) > d_1(H)+d_2(H)+1 = 5.$ If $p \geq 2, r=s=0,$ then $2p+4 = \lambda_1(H)+\lambda_2(H) > d_1(H)+d_2(H)+1 = 2p+1.$ The remaining cases are similar to the ones of the Theorem \ref{th3} and equality holds when $G = S_{n}.$
\end{proof}

\section{The general case}\label{sec:Section2and3extendedQ}
%%%%%%%%%%%%%%%%%%%%%%%%%%%%%%%%%%%%%%%%%%%%%%%%%%%%%%

%We will use the following preliminary result.

%\begin{lemma}({\cite{BH12}}) \label{lemma}
%Let $A$ be a symmetric matrix of order $n$, and suppose $\rho$ is a partition of
%$\{1,\ldots,n\}$ such that the corresponding partition of $A$ is equitable with quotient matrix $B$.
%Then the spectrum of $B$ is a sub(multi)set of the spectrum of $A$, and all corresponding
%eigenvectors of $A$ are in the column space of the characteristic matrix $C$ of $\rho$ (this means that the entries of the eigenvector are constant on each partition class $U_i$). The remaining eigenvectors of $A$ are orthogonal to the columns of $C$ and the corresponding
%eigenvalues remain unchanged if the blocks $A_{i,j}$ are replaced by %$A_{i,j} + c_{i,j}J$ for certain
%constants $c_{i,j}$ (as usual, $J$ is the all-one matrix).
%\end{lemma}
In this section, using a different approach than in Section \ref{sec:twoevs}, we obtain several sharp bounds on the sum of the largest signless Laplacian eigenvalues. We will see that for the case of $q_1(G)+q_2(G)$, our bounds from Section \ref{sec:twoevs} and \ref{sec:Section2and3extendedQ} are incomparable.

From inequality (\ref{eq:1HaemersSeminar2}) we have 
\begin{align}
	\sum_{i=1}^m q_i(G) \geq \sum_{i=1}^m d_i(G) \label{eq1}.
\end{align}
for $ 1 \leq m \leq n$. If $m=n$ then we have equality in \eqref{eq1}, because both terms correspond to the trace of $Q$.
Similarly we have:
\begin{align}
	\sum_{i=1}^m q_{n-m+i}(G) \leq \sum_{i=1}^m d_{n-m+i}(G).
\end{align}

For a vertex set $U \subset V$ such that $|U|\,=\, m$, write $\partial(U)$ as the set of vertices in $\overline{U} = V\setminus U$ with at least one adjacent vertex in $U$, and  $\partial(U,\overline{U})$ as the set of edges connecting vertices in $U$ with vertices in $\overline{U}.$ 
The next result shows that the above bounds can be pushed further by using a mix of two types of eigenvalue interlacing (Cauchy and quotient matrix interlacing).

\begin{theorem} \label{t1}
 Let $G$ be a connected graph on $n$ vertices. For any given vertex subset $U = \{u_1,\ldots,u_m\}$
with $0 <m<n$, we have
\begin{align}
	\sum_{i=1}^{m+1} q_{n-i}(G) \leq \sum_{u\in U} d_u + \frac{\sum_{\overline{u} \in \overline{U}} d_{\overline{u}}+2E[\overline{U}]}{n-m} \leq \sum_{i=1}^{m+1} q_i(G). \label{eq3}
\end{align}
\end{theorem}

\begin{proof}
Let $U \subset V$ such that $|U|=m$ where $0 < m < n$. Consider the partition of the vertex set $V$ into $m+1$ parts such that $U_i = \{u_i\}$ for $u_i \in U$, $i=1,\ldots,m$, and $U_{m+1}=\overline{U}$. Then, the corresponding quotient matrix of this partition is

\begin{equation*}
	B' = 
	\left[\begin{array}{c c c | c}
		        &         &        & b'_{1,m+1} \\
		        & Q_U    &        & \vdots \\
		        &         &        & b'_{m,m+1}  \\
		        \hline
		b'_{m+1,1} & \cdots & b'_{m+1,m} & b'_{m+1,m+1} 
	\end{array}\right],
\end{equation*}
where $Q_U$ is the principal submatrix of $Q$, with rows and columns indexed by the vertices in $U$, $b'_{i,m+1}=|\partial(U_i,\overline{U})|$, $b'_{m+1,i}=\frac{|\partial(U_i,\overline{U})|}{n-m}$, and $b'_{m+1,m+1} = (\sum_{\overline{u} \in \overline{U}} d_{\overline{u}} +2E[\overline{U}])/(n-m)$. Note $\mu'_1 \geq \mu'_2 \geq \cdots \geq \mu'_{m+1}$ are the eigenvalues of $B'$, then 

\begin{equation}
	\sum_{i=1}^{m+1} \mu'_i = tr B' = \sum_{u \in U} d_u + \frac{\sum_{\overline{u} \in \overline{U}} d_{\overline{u}} +2E[\overline{U}]}{n-m}.
\end{equation} 

From Theorem \ref{the:interlacinghaemers}, we have that $q_i(G) \geq \mu'_i \geq q_{n-m-1+i}(G),$ for $i= 1,2, \ldots,m$ and the result follows.
\end{proof}

% \begin{align*}
%     \sum_{u\in U} d_u + \frac{\sum_{\overline{u} \in \overline{U}} d_{\overline{u}}+2E[\overline{U}]}{n-m} & \leq d_1 + d_2 + 1\\
%     d_1 + \frac{(n-1)}{n-1}&\leq d_1 + d_2 +1
% \end{align*}
% \begin{align*}
%     \sum_{u\in U} d_u + \frac{\sum_{\overline{u} \in \overline{U}} d_{\overline{u}}+2E[\overline{U}]}{n-m} & \leq d_1 + d_2 + 1\\
%     d_2 +  \frac{d_1+3(n-2)}{n-1} &\leq d_1 + d_2 +1
% \end{align*}

\begin{remark}
Consider $G$ as a $r$-regular connected graph and let $m=1$ in Theorem \ref{t1}. Take an arbitrary vertex $u_{1} \in U$. In this case the lower bound provided in Theorem \ref{t1} is equal to $2r + r(n-2)/(n-1)$ which is better than the lower bound provided in Theorem \ref{th3} since $n \geq 3$. However, in general, those bounds are incomparable. For instance, for the star $S_n$, the lower bound provided in Theorem \ref{th3} is better than the one in Theorem \ref{t1}.
\end{remark}

Next we investigate the tightness of the bounds from Theorem \ref{t1}.

\begin{proposition}\label{propo:equalitythm}
Let $H$ be the subgragh of $G$ induced by $\overline{U}$, and let $q'_1 \geq \cdots \geq q'_{n-m}$ be the signless Laplacian eigenvalues of $H$. Define $b'=\frac{|\partial(U,\overline{U})|}{n-m}$. 
Then equality holds on the right hand side of \eqref{eq3} if and only if each vertex of $U$ is adjacent
 	to all or no vertices of $\overline{U}$, and $q_{m+1} = q'_1 + b' = d_{\overline{u}},$ where $\overline{u} \in \overline{U}.$
\end{proposition}

\begin{proof}
Suppose equality holds on the right hand side of \eqref{eq3}. Then
\begin{equation*}
	\sum_{i = 1}^{m+1} q_i = \sum_{i = 1}^{m+1} q'_i \textrm{  and  } q_i \geq q'_i 
\end{equation*}
so $q_i=q'_i$ for $i=1,\ldots,m+1$, therefore the interlacing is tight and hence the partition of $G$ is almost equitable. Each vertex in $U$ is adjacent to all or $0$ vertices of $\overline{U}$ since each block should have a constant row and column sum, so each vertex in $\overline{U}$ has a constant number of neighbors in $U$ that is $b'$. By calculation it can be deduced that $H$ is $q'_1$-regular. Now by use of Lemma \ref{lemma} we have that the eigenvalues of $Q$ are $  q'_1,\ldots,q'_{m+1}$ together with eigenvalues of $Q$ with an eigenvector orthogonal to the characteristic matrix $C$ of the partition. These eigenvalues and eigenvectors remain unchanged if $Q$ is changed into
\begin{equation*}
	\overset{\sim}{Q} = 
	\left[\begin{array}{cc}
	  O  &    O   \\
	  O  &   Q_{\overline{U}}+b'I  \\	
	\end{array}\right].
\end{equation*}
The considered common eigenvalues of $\overset{\sim}{Q}$ and $Q$ are $q'_1+b'\geq\cdots\geq q'_{n-m-1}+b'$. So $Q$ has eigenvalues $q_1(=q'_1)\geq\cdots\geq q_{m+1}(=q'_{m+1})$, and $q'_1+b'\geq\cdots\geq q'_{n-m-1}+b'$. Hence, we have
$q_{m+1} = q'_1+b'$. Conversely, if the partition of $G$ is almost equitable, $Q$ has eigenvalues $q'_1 \geq \cdots \geq q'$, $q'_1+b' \geq \cdots \geq q'_{n-m-1}+b$. Since $q_{m+1} = q'_1+b'$, it follows that $q'_i = q_i$ for $i = 1,\ldots,m$ (tight interlacing), therefore equality holds on the right-hand side of \eqref{eq3}.
\end{proof}

\begin{theorem}\label{thm:1bis}
 Let $G$ be a connected graph on $n$ vertices. For any given vertex subset $U = \{u_1,\ldots,u_m\}$
with $0 <m<n$, we have

\begin{align}
	\sum_{i=1}^{m} q_{n-i+1} < \sum_{u\in U} d_u + \frac{4E[\bar{U}]}{n-m} < \sum_{i=1}^{m} q_i. \label{eq3a}
\end{align}
\end{theorem}

\begin{proof}
	Take matrix $B'$ and its eigenvalues similarly as Theorem \ref{t1}. By Theorem \ref{t1} we have,
	
	\begin{equation}
	\sum_{i=1}^{m+1} \mu'_i = tr B' = \sum_{u \in U} d_u + \frac{\sum_{\overline{u} \in \overline{U}} d_{\overline{u}} +2E[\overline{U}]}{n-m}.
\end{equation} 

\begin{claim}\label{claim}
$\mu'_{m+1}<\frac{|\partial(U,\bar{U})|}{n-m}<\mu'_1$.\end{claim}
\begin{proof}[Proof of Claim \ref{claim}.]
Let $\bar{E}=\{e_1, \ldots, e_k\}$ be the set of all edges of $G$ that have at least one endpoint in $U$. Define a $(m+1) \times k$ matrix $D$ as follows:
$$
D_{ij} = \left\{
\begin{array}{ll}
	1, & \text{ if } e_j \in \bar{E} \text{ is incident to at least one vertex in } U_i \\
	0, & \text{otherwise}.
\end{array}
\right.
$$
Also, consider the $(m+1)  \times (m+1)$ diagonal matrix $P$ with diagonal entries equal to $1$ except the last one equal to $\frac{1}{ \sqrt[]{n-m}}$:
\[
P =
\begin{bmatrix}
	1 & & &\\
	& \ddots && \\
	& & 1&\\
	& & & \frac{1}{ \sqrt[]{n-m}}
\end{bmatrix}.
\]
Now let $\tilde{B}=DD^{T}$. It is easy to see that $B'=P^{2}\tilde{B}$, and it follows that

%by multiplying $P^{-1}$ from left and $P$ from right we obtain:
\begin{equation*}
	P^{-1}B'P=P\tilde{B}P=A,
\end{equation*} 
which means $B'$ and $A$ are similar, so $\mu'_1$ and $\mu'_{m+1}$ are the largest and the smallest eigenvalues of $A$, respectively. Then by the Rayleigh principle we find that:
\begin{equation*}
	\mu'_{m+1}\leq\frac{x^{T}Ax}{x^{T}x}=\frac{x^{T}P^{T}DD^{T}Px}{x^{T}x}=\frac{\vert\vert D^{T}Px\vert\vert^2}{\vert\vert x\vert\vert^2}\leq\mu'_1
\end{equation*} 
Then, taking $x=(0,\ldots,0,1)$ the claim follows.
Moreover, we show that both inequalities are strict. In Rayleigh principle a necessary condition for equality to be hold in both inequalities is that $x$ must be an eigenvector of $A$. Assume that $x$ is an eigenvector of $A$. It means that all entries of the last column of $A$ except the last entry must be $0$. Similarly, all entries of the last column of $\tilde{B}$ except the last entry must be $0.$ Since $P$ is a matrix such that by being multiplied from left(right) to $\tilde{B}$ it will only multiply the last row(column) of $\tilde{B}$ by $1/\sqrt{n-m}$. On the other hand, $G$ is connected, thus there exists an edge $e_1 \in \bar{E}$, which has endpoints in $U_{m+1}$ and $U_1$. Thus the associated entry is always greater than $0$ which is a contradiction. Thus both upper and lower bounds are strict.
\end{proof}

Now, by multiplying $-1$ to each side of the claim inequality, and adding up $\sum_{i=1}^{m+1} \mu'_i$, we obtain
\begin{align*}
\sum_{i=2}^{m+1} \mu'_i &<\sum_{i=1}^{m+1} \mu'_i+\frac{-\partial(U,\bar{U})}{n-m}\\
&=\sum_{u \in U} d_u + \frac{\sum_{\bar{u} \in \bar{U}} d_i +2E[\bar{U}]}{n-m}+\frac{-\partial(U,\bar{U})}{n-m}\\
&=\sum_{u\in U} d_u + \frac{4E[\bar{U}]}{n-m}<\sum_{i=1}^{m} \mu'_i,
\end{align*}
and applying interlacing, we get $q_i \geq \mu'_i \geq q_{n-m-1+i}$, for $i=1,\ldots,m+1$. Then \eqref{eq3a} follows from observing $\sum_{i=1}^{m} \mu'_i\leq\sum_{i=1}^{m} q_i$ and $\sum_{i=1}^{m} q_{n-i+1}\leq\sum_{i=2}^{m+1} \mu'_i$. 
\end{proof}

As a consequence of Theorem \ref{t1} we have the following two corollaries.

\begin{corollary}\label{propo:applications1}
Let $G$ be a connected $k$-regular graph on $n = |V |$ vertices, having adjacency matrix $A$ with eigenvalues $k = \gamma_1 \geq \gamma_2 \geq \cdots \geq \gamma_n$. For any given partition $(U,\bar{U})$ such that $|U|=m$, it holds
\begin{align*}
	 \frac{2E[\bar{U}]}{n-m} \leq \sum_{i=1}^{m+1} \gamma_i. \label{eq3}
\end{align*}
\end{corollary}
\begin{proof}
Since $d_v=k$ for any vertex $v \in V(G)$, Theorem \ref{t1} implies that
\begin{equation*}\label{eq3}
	(m+1)k + \frac{2E[\bar{U}]}{n-m} \leq \sum_{i=1}^{m+1} q_i. 
\end{equation*}
Note that $q_i=\gamma_i + k$ for $ 1 \leq i \leq n$ and thus the result follows.
\end{proof}
 
 Note that for regular complete multipartite graphs, the bound in Corollary \ref{propo:applications1} holds with equality if and only if $m = n-1$.

% \textcolor{blue}{
% \begin{example}\label{ex:coroadjacency}
% This example illustrates that the bound from Corollary \ref{propo:applications1} is tight. Consider the $K_{k,\ldots,k}$ regular complete multipartite graph with t classes of
% size k, so n = tk and with regularity degree $(t-1)k$. The eigenvalues of its adjacency matrix are
% \begin{equation*}
% 	\{-k^{\,t-1},0^{\,t(k-1)}, k(t-1)\}.
% \end{equation*}
% Observe that if $U$ is the union of $r<t$ partitions, then $U$ contains $rk$ vertices and also $\overline{U}$ contains $(t-r)k$ vertices. It is easy to see that the left side of the inequality \eqref{eq3b} equals $(t-r-1)k$ and the right side equals 
% \begin{equation*}
% 	(t-1)k+\underbrace{0+\cdots+0}_{p}+\underbrace{(-k)+\cdots+(-k)}_{s}
% \end{equation*}
% It can be seen that the inequality is strict if $s=0$, on the other hand if $0<s$ then $p=(k-1)t$. Also, $s < t$. Now note that,
% \begin{equation*}
% 			rk=m=(k-1)t+s
% \end{equation*}
% and after simplifying the equation we get,
% \begin{equation*}
% 	t-s=k(t-r)
% \end{equation*}
% and since $k$ is a natural number greater than 1, so $s<r$ which shows that the inequality, in this case, is strict as well.
% \end{example}
% }

\begin{corollary}\label{propo:applications2}
Let $G$ be a connected graph on $n$ vertices. If $I$ is an independent set of cardinality $\alpha > 1$, then
\begin{align*}
    \frac{\alpha}{\alpha-1} \sum_{i=n-\alpha+2}^{n} q_i(G)  \leq \sum_{\overline{u} \in I} d_{\overline{u}}.
\end{align*}
\end{corollary}

\begin{proof}
Let $U = \overline{I}$. From Theorem \ref{t1} and noting that $I$ has no edges we have
\begin{equation*}\label{eq3}
     \sum_{u \in \overline{I}} d_u + \frac{\sum_{\overline{u} \in I} d_{\overline{u}}}{\alpha} \leq \sum_{i=1}^{n-\alpha+1} q_i(G). 
\end{equation*}
Now, note that 
$$\sum_{u\in \overline{I}} d_u = 2E(G) - \sum_{\overline{u} \in I} d_{\overline{u}} $$
and 
\begin{align*}
\sum_{i=1}^{n-\alpha+1} q_i(G) = 2E(G) - \sum_{u=n-\alpha+2}^n q_i(G).
\end{align*}
So the result follows.
\end{proof}

 The next result is a $Q$-analog version and an extension of a bound by Grone and Merris for the sum of the largest Laplacian eigenvalues \cite{GM1994}.
 
\begin{theorem}\label{thm:latestimprovement}
 Let $G$ be a connected graph of order $n \geq 3$. Given $a$ vertex subset $U \subset V $, with $m = |U| < n$ such that $G[U]=(U, E[U])$ be its induced subgraph. Then
 
 \begin{equation*}
 	\sum_{i=1}^{m+1} q_i \geq \sum_{u \in U} d_u + m - |E[U]|. \label{eq5}
 \end{equation*}
 
\end{theorem}
\begin{proof}
Consider an orientation of $G$ with all edges in $E(U,\overline{U})$ oriented from $U$ to $\overline{U}$, and every vertex in $U\setminus\partial\overline{U}$ having some outgoing arc (this is always possible as $G$ is connected). Let $Q$ be the corresponding oriented incidence matrix of G, and let $(D)_{ij} = |(Q)_{ij}|$. Write
$D = [D_1 \, D_2]$, where $D_1$ corresponds to $E[U] \cup E(U,\overline{U})$, and $D_2$ corresponds to $E[\overline{U}]$.
Consider the matrix $M' = D^{\top}D$ , with entries $(M)_{ii} = 2$, $(M)_{ij} = 1$ if the arcs $e_i, e_j$ are
incident to the same vertex,
and $(M)_{ij} = 0$ if the corresponding edges are disjoint, and define $M'_1 = D_1^{\top}D_1$. Then
$M'$ has the same nonzero eigenvalues as $Q = DD^{\top}$, the signless Laplacian matrix of $G$, and $M'_1$
is a principal submatrix of $M'$. For every vertex $u \in U$, let $E_u$ be the set of outgoing
arcs from $u$. Then $\{E_u | u \in U\} $ is a partition of $E[U] \cup E(U,\overline{U})$. Consider the quotient
matrix $B_1 = (b_{ij})$ of $M'_1$ with respect to this partition. Then, $b_{uu} = d^{+}_{u} + 1$ for each
$u \in U$. Let $\mu_1 \geq \mu_2 \geq \cdots \geq \mu_m$ be the eigenvalues of $B_1$, then
\begin{equation*}
	tr B_1 = \sum_{i=1}^{m} \mu_i = \sum_{u \in U} d_u^{+} + m = \sum_{u \in U} d_u - |E[U]| + m
\end{equation*}
and the result follows since the eigenvalues of $B_1$ interlace those of $M'_1$, which in turn interlace those of $M'$.
\end{proof}

Note that Theorem \ref{thm:latestimprovement} can be improved by considering the partition $\mathcal{P} = \{E_u|u\in U\} \cup \{E[\overline{U}]\}$ of the whole edge set of $G$:

\begin{theorem}
 Let $G$ be a connected graph of order $n \geq 3$. Given a vertex subset $U \subset V $, with $m = |U| < n$, let $G[U]=(U, E[U])$ and
$G[\overline{U}]$ be the corresponding induced subgraphs. Let $q'_1$ be the largest signless Laplacian eigenvalue
of $G[\overline{U}]$ Then
\begin{align}
\sum_{i=1}^{m+1} q_i \geq \sum_{u \in U} d_u + m - |E[U]| + q'_1.
\end{align}
\end{theorem}

\begin{proof}
First observe that the signless Laplacian matrix of $G[\overline{U}]$ is $D_2D_2^{\top}$, and therefore $q'_1$ is also
the largest eigenvalue of $D_2^{\top}D_2$. Next we apply interlacing to an $(m+1)\times (m+1)$ quotient matrix $B = S^{\top}MS$, which is defined slightly different than before. The first $m$ columns of $S$ are the normalized characteristic vectors of $E_u$ (as before), but the last column of $S$
equals $\left[\begin{array}{cc}
	0 & v
\end{array}\right]^\top$, where $v$ is a normalized eigenvector of $D_2^{\top}D_2$ for the eigenvalue $q'_1$. Then
$b_{m+1,m+1} = (D_2v)^{\top}D_2v = q'_1$, and we find $tr B = \sum_{u \in U} d_u + m - |E[U]| + q'_1$. 
\end{proof}

%%%%%%%%%%%%%%%%%%%%%%%%%%%%%%%%%%%%%%%%%%%%%%%%%%%%%%
\subsection*{Acknowledgements} 
%%%%%%%%%%%%%%%%%%%%%%%%%%%%%%%%%%%%%%%%%%%%%%%%%%%%%%
The research of Aida Abiad is partially supported by the FWO grant 1285921N. Leonardo de Lima has been supported by CNPq Grant 315739/2021-5, and Carla Oliveira has also been supported by CNPq Grant 30458/2020-0. The authors thank Ali Mohammadian for his valuable comments on Section \ref{sec:twoevs}.

A preliminary version of this paper appeared in the Proceedings of the 14th Cologne-Twente Workshop on Graphs and Combinatorial Optimization (CTW16) in \emph{Electronic Notes in Discrete Mathematics} 55 (2016), 173--176.

%%%%%%%%%%%%%%%%%%%%%%%%%%%%%%%%%%%%%%%%%%%%%%%%%%%%%%

%%%%%%%%%%%%%%%%%%%%%%%%%%%%%%%%%%%%%%%%%%%%%%%%%%%%%
\end{document}